\documentclass[12pt]{article}
\usepackage{amssymb,amsfonts,amsmath,amsthm,dsfont,hyperref}
\usepackage[a4paper, portrait, margin=28mm]{geometry}
\usepackage[lite,alphabetic]{amsrefs}

\newcommand{\1}{\mathds{1}}
\newcommand{\0}{\mathds{O}}

\newcommand{\R}{\mathbb{R}}

\newcommand{\N}{\mathbb{N}}

\newcommand{\8}{\infty}

\newcommand{\Co}{\mathcal{C}}
\newcommand{\Fo}{\mathcal{F}}

\newcommand{\Po}{\mathcal{P}}

\newcounter{dummy} \numberwithin{dummy}{section}
\newtheorem{theorem}[dummy]{Theorem}
\newtheorem{lemma}[dummy]{Lemma}

\newtheorem{proposition}[dummy]{Proposition}
\newtheorem{corollary}[dummy]{Corollary}
\newtheorem{question}[dummy]{Question}
\theoremstyle{remark}
\newtheorem{remark}[dummy]{Remark}
\newtheorem{example}[dummy]{Example}
\usepackage{scalerel}

\begin{document}

\title{Composition of locally solid convergences}
\author{Eugene Bilokopytov\footnote{Email address bilokopy@ualberta.ca, erz888@gmail.com.}}
\maketitle

\begin{abstract}
We carry on a more detailed investigation of the composition of locally solid convergences as introduced in \cite{ectv}, as well as the corresponding notion of idempotency considered in \cite{erz}. In particular, we study the interactions between these two concepts and various operations with convergences. We prove associativity of the composition and show that the adherence of an ideal with respect to an idempotent convergence is equal to its closure. Some results from \cite{kt} about unbounded modification of locally solid topologies are generalized to the level of locally solid idempotent convergences. A simple application of the composition allows us to answer a question from \cite{ectv} about minimal Hausdorff locally solid convergences. We also show that the weakest Hausdorff locally solid convergence exists on an Archimedean vector lattice if and only if it is atomic.

\emph{Keywords:} Vector lattices, Net convergence spaces;

MSC2020 46A19, 46A40, 46B42, 46E25, 54A20.
\end{abstract}

\section{Introduction}

This note is a continuation of the study of locally solid convergence structures on vector lattices. In the previous work \cite{ectv} a definition of a composition of locally solid convergences was proposed. The corresponding notion of an idempotent convergence was considered in \cite{erz}. Here we add more details pertaining to these two concepts. In particular, we show (Theorem \ref{product}) that if $E\subset F$ is an ideal of a vector lattice, and $\eta,\theta$ are two locally solid convergences, then $\overline{E}^{1,\eta\theta}=\overline{\overline{E}^{1,\eta}}^{1,\theta}$, which somewhat justifies the term ``composition''; we also establish in Corollary \ref{ass} that composition is associative. In Corollary \ref{mon} and Theorem \ref{intd} we explore the topic of convergences which share decreasing null nets. We show (propositions \ref{unbo}, \ref{pullb} and \ref{choq}) that unbounded modification with respect to an ideal, a pull-back along a positive operator and the Choquet modification are ``submultiplicative'', and in particular preserve idempotent convergences (Example \ref{idempotent}). Numerous examples of idempotent convergences are provided in Subsection \ref{idem}.\medskip

We also present some applications of the concepts of composition and idempotency of locally solid convergences. In Section \ref{weake} we look at the minimal Hausdorff locally solid convergences. We show (Theorem \ref{weakest}) that the Choquet modification of unbounded order convergence is a minimal Hausdorff locally solid convergence on an Archimedean vector lattice if and only if the latter is atomic, thus answering a question from \cite{ectv}. Furthermore, another condition equivalent to atomicity is existence of the weakest Hausdorff locally solid convergence. Here the relevance of the notion of composition is a simple result (Lemma \ref{weaker}) which reduces the problem from the level of locally solid \textbf{linear} convergences to that of locally solid \textbf{additive} convergences, which are easier to deal with. In Section \ref{unb} we focus on the unbounded modification, and generalize some results from \cite{kt} by replacing locally solid topologies with idempotent locally solid convergences. In particular, Theorem \ref{unbc} gives a fairly comprehensive characterization of when one unbounded convergence is stronger than another.\medskip

We remark here that a notion, similar to powers of relative uniform convergence was considered in \cite{fremlin}, but the author is not aware of any further developments in that direction.

\subsection{Preliminaries}

We refer to \cite{erz} and \cite{ectv} for the theory of locally solid convergences and to \cite{bb} and \cite{dow} for more background on convergence spaces. Note that the source \cite{erz} studies additive convergences (in which the addition and taking the additive inverse are continuous), while only linear convergences (in which scalar multiplication is also continuous) are considered in \cite{ectv}. We need a slightly broader class of additive convergences, due to its somewhat higher flexibility. In particular, we will often rely on the following construction.

\begin{example}\label{propadd}
Let $F$ be a linear space and let $E\subset F$ be a subspace endowed with an additive convergence $\eta$. This convergence can be extended to $F$ by declaring a net $\left(f_{\alpha}\right)_{\alpha\in A}\subset F$ null if its tail is contained in $E$ and is $\eta$-null there, and then defining $f_{\alpha}\to f$ if $f_{\alpha}-f\to 0_{F}$. One can show that the convergence $\eta_{F}$ defined this way is indeed an additive convergence structure, and $\eta_{F}$ is Hausdorff iff $\eta$ is. Furthermore, $\eta_{F}$ is the strongest additive convergence on $F$ whose restriction to $E$ is weaker than $\eta$ (it is the final additive convergence with respect to the inclusion of $E$ into $F$). Note that $E$ is $\eta_{F}$-closed in $F$. Indeed, if $\left(e_{\alpha}\right)_{\alpha\in A}\subset E$ converges to $f\in F$, there is $\alpha$ such that $f-e_{\alpha}\in E$, hence $f\in E$. If $E\varsubsetneq F$ then $\eta_{F}$ is not linear, since $\frac{1}{n}f\not\to 0_{F}$ whenever $f\in F\backslash E$.
\qed\end{example}

In the sequel $F$ is a vector lattice. We will not reproduce all the definitions necessary to define locally solid convergences. Instead, we will operate through the following result.

\begin{theorem}[\cite{erz}, Theorem 2.1 and \cite{ectv}, Proposition 3.4]\label{locs}
Assume that $\to 0_{F}$ is a non-trivial adjudicator on $F_{+}$ to $0_{F}$ which satisfies the following properties:
\begin{itemize}
\item A quasi-subnet of a net convergent to $0_{F}$ converges to $0_{F}$;
\item If $\left(f_{\alpha}\right)_{\alpha\in A}\subset F_{+}$ is convergent to $0_{F}$, then $g_{\alpha}\to 0_{F}$, where $0_{F}\le g_{\alpha}\le f_{\alpha}$, for every $\alpha\in A$;
\item If $f_{\alpha}\to 0_{F}$ and $g_{\alpha}\to 0_{F}$, then $f_{\alpha}+g_{\alpha}\to 0_{F}$.
\end{itemize}
Then, the convergence defined by $f_{\alpha}\to f$ if $\left|f_{\alpha}-f\right|\to 0_{F}$ is a locally solid additive convergence. It is linear iff $\frac{1}{n}f\to 0_{F}$, for every $f\in F_{+}$, and Hausdorff if $0_{F}<f\not\to0_{F}$.
\end{theorem}

We call a convergence $\eta$ \emph{stronger} than $\theta$ (denoted $\eta\ge\theta$) if convergence with respect to $\eta$ implies convergence with respect to $\theta$. The set of all (additive or linear) locally solid convergences is a complete lattice: if $\Theta$ is a collection of convergences, a net converges with respect to $\bigvee\Theta$ if it converges with respect to each member of $\theta$. On the other hand, if $\Theta$ is directed downward, then a net converges with respect to $\bigwedge\Theta$ if it converges with respect to a member of $\theta$. In particular, the infimum of a decreasing net of Hausdorff convergences is Hausdorff.

\begin{example}
The discrete convergence, in which only eventually constant nets converge, is locally solid additive.
\qed\end{example}

It is easy to see (by Theorem \ref{locs} or directly) that if $E\subset F$ is an ideal and $\eta$ is a locally solid additive convergence on $E$, then its extension $\eta_{F}$ is a locally solid additive convergence on $F$. This means that we can always view our convergences defined on $F$, thus circumventing the issue which appears in \cite{klt}.

\begin{example}\label{lru}
For $e\in F_{+}$ the principal ideal $F_{e}$ is endowed with the norm $\|\cdot\|_{e}$ defined by $\|f\|_{e}=\bigwedge\left\{\alpha\ge 0:~ \left|f\right|\le\alpha e\right\}$. We will extend the convergence determined by $\|\cdot\|_{e}$ to the entire $F$ (and keep the same notation). The resulting convergence is locally solid additive, and \emph{relative uniform} convergence is precisely $\bigwedge\limits_{e\in F_{+}}\|\cdot\|_{e}$.
\qed\end{example}

\begin{example}\label{aw}
Let $F$ be a normed lattice. If we ``feed'' the weak convergence to $0_{F}$ on $F_{+}$ into Theorem \ref{locs}, we obtain the \emph{absolute weak} topology $\mathrm{aw}$. This is the weakest locally solid convergence stronger than the weak topology (which fails to be locally solid in most cases). In particular, $\mathrm{aw}$ is weaker than the norm topology. On the other hand, Dini's theorem (the proof of \cite[Proposition 1.4.1]{mn} works for nets) asserts that a decreasing weakly null net on a normed lattice is norm-null, and so $\mathrm{aw}$ agrees with the norm convergence on monotone nets.
\qed\end{example}

A significant role in this article is played by monotone nets. Note that a directed upward subset of $F$ may be viewed as an increasing net indexed by itself. Similarly, if $G\subset F$ is directed downward, it can be viewed as a decreasing net indexed by $-G$. In either case, the notation $G\to f$ means that $G$ converges to $f$ as a net indexed by $G$ or $-G$. See \cite[]{ectv} for some facts about the convergence of monotone nets. Here we reproduce one of them (the proof works for the additive case).

\begin{lemma}[\cite{ectv}, Proposition 3.9]\label{mono}
Let $\eta$ be a locally solid additive convergence on $F$. Then, a decreasing net in $F_{+}$ $\eta$-converges to $0_{F}$ as soon as it contains (as a subset) a $\eta$-null net.
\end{lemma}

It is known (\cite[Proposition 2.7]{vw}) that the adherence of a solid / ideal with respect to an additive locally solid convergence is a set of the same type. We denote the adherence of $A\subset F$ with respect to convergence $\eta$ by $\overline{A}^{1}_{\eta}$, and the closure -- by $\overline{A}_{\eta}$. If $\eta$ is clear from the context, it will be dropped from these notations.

We will call $G\subset F_{+}$ \emph{positively solid}, if it is a lower set in $F_{+}$, i.e. if $0_{F}\le f\le g\in G$ $\Rightarrow$ $f\in G$. If on top of that $G$ is $\vee$-closed (equivalently, directed upward), we call it a \emph{lattice-ideal} of $F_{+}$. For example, if $H\subset F$ is an ideal, then $H_{+}$ is a lattice ideal of $F_{+}$. Dually, a $\wedge$-closed upper subset of $F_{+}$ is called \emph{co-ideal} (a more standard term is \emph{filter}, but we would like to avoid confusion with a filter \textbf{on} $F$, which is a subset of $\Po\left(F\right)$). We will view lattice-ideals as increasing nets, and co-ideals -- as decreasing nets.

\begin{lemma}\label{psola}Let $G\subset F_{+}$ be positively solid and let $g\in F_{+}$. Then, $g\in \overline{G}^{1}$ iff there is $\left(g_{\alpha}\right)_{\alpha\in A}\subset G\cap\left[0_{F},g\right]$ such that $g_{\alpha}\to g$ and iff there is $\left(g_{\alpha}\right)_{\alpha\in A}\subset G$ such that $\left(g-g_{\alpha}\right)^{+}\to 0_{F}$.
\end{lemma}
\begin{proof}
Clearly, the second condition is the strongest, whereas the last one is the weakest. On the other hand, if $\left(g_{\alpha}\right)_{\alpha\in A}\subset G$ is such that $\left(g-g_{\alpha}\right)^{+}\to 0_{F}$, then $\left(g-g_{\alpha}\right)^{+}=g-g\wedge g_{\alpha}$ with $g\wedge g_{\alpha}\in G\cap\left[0_{F},g\right]$, for every $\alpha$. It then follows that $g=\lim\limits_{\alpha}\left(g\wedge g_{\alpha}\right)$.
\end{proof}

Combining this result with Lemma \ref{mono} we get the following.

\begin{corollary}\label{psola1}If $G\subset F_{+}$ a lattice-ideal, then $g\in \overline{G}^{1}$ iff $G\cap\left[0_{F},g\right]\to g$.
\end{corollary}

Analogous result holds for co-ideals.

\begin{corollary}\label{coid}If $G\subset F_{+}$ is a co-ideal, then $G\to 0_{F}$ iff $0_{F}\in \overline{G}^{1}$.
\end{corollary}

\begin{proposition}\label{adh}
If $G,H\subset F_{+}$ are positively solid, then $\overline{G\cap H}^{1}=\overline{G}^{1}\cap \overline{H}^{1}$. Same is true for solid subsets of $F$.
\end{proposition}
\begin{proof}
Monotonicity of adherence implies $\overline{G\cap H}^{1}\subset \overline{G}^{1}, \overline{H}^{1}$. On the other hand, if $f\in \overline{G}^{1}\cap \overline{H}^{1}$, there are nets $\left(g_{\alpha}\right)_{\alpha\in A}\subset G$ and $\left(h_{\beta}\right)_{\beta\in B}\subset H$ which converge to $f$. Due to continuity of $\wedge$ we conclude that $G\cap H\ni g_{\alpha}\wedge h_{\beta}\to f$, hence $f\in \overline{G\cap H}^{1}$.

We now consider the case when $G,H\subset F$ are solid and $f\in \overline{G}^{1}\cap \overline{H}^{1}$. It is then easy to see that $\left|f\right|\in\overline{G_{+}}^{1}\cap \overline{H_{+}}^{1}=\overline{G_{+}\cap H_{+}}^{1}\subset \overline{G\cap H}^{1}$. As the latter set is solid, we conclude that $f\in \overline{G\cap H}^{1}$.
\end{proof}

An analogous fact is also true for co-ideals.

\begin{proposition}\label{sadh}
The adherence of a countably generated ideal is equal to its sequential adherence.
\end{proposition}
\begin{proof}
Assume that an ideal $E\subset F$ is generated by an increasing sequence $\left\{e_{n}\right\}_{n\in\N}$ and fix $0_{F}\le f\in\overline{E}^{1}$, so that  $E\cap\left[0_{F},f\right]\to f$. We have that $E=\bigcup\limits_{n\in\N}F_{e_{n}}$, and so for every $e\in E$ there are $m,n\in\N$ such that $e\le me_{n}$. Taking $m\vee n$ leads to the conclusion that $\left\{f\wedge ne_{n}\right\}_{n\in\N}$ is co-final in $E\cap\left[0_{F},f\right]$, hence converges to $f$.
\end{proof}

\section{Composition of locally solid convergences}

Throughout the section, $F$ is a vector lattice. Let $\eta$ and $\theta$ be locally solid additive convergences on $F$. Following \cite[Section 10]{ectv} we define a new convergence $\eta\theta$ as follows: for a net $\left(f_{\alpha}\right)_{\alpha\in A}\subset F_{+}$ we will say that $f_{\alpha}\xrightarrow[\alpha]{\eta\theta} 0_{F}$ if there is a \emph{controlling} $\theta$-null net $\left(g_{\beta}\right)_{\beta\in B}\subset F_{+}$ such that $\left(f_{\alpha}-g_{\beta}\right)^{+}\xrightarrow[]{\eta} 0_{F}$, for every $\beta$. Let us summarize some basic facts about $\eta\theta$.

\begin{proposition}\label{produc}Let $\eta$ and $\theta$ be locally solid additive convergences on $F$. Then:
\item[(i)] $\eta\theta$ is a locally solid additive convergence, which is weaker than $\eta$.
\item[(ii)] Every monotone $\theta$-convergent net is $\eta\theta$-convergent. If $\theta$ is order continuous, so is $\eta\theta$.
\item[(iii)] $\eta\theta$ is Hausdorff iff both $\eta,\theta$ are.
\item[(iv)] If one of $\eta,\theta$ is linear, then $\eta\theta$ is linear.
\end{proposition}
\begin{proof}
(i)-(iii) were proven in \cite[Propositions 10.11-10.13]{ectv} for locally solid linear convergences, but the proof goes through for the additive case.

(iv): Let $e\in F_{+}$. If $\eta$ is linear, $\frac{1}{n}e\xrightarrow[]{\eta} 0_{F}$, hence $\frac{1}{n}e\xrightarrow[]{\eta\theta} 0_{F}$, according to (i). If $\theta$ is linear, $\frac{1}{n}e\xrightarrow[]{\theta} 0_{F}$, hence $\frac{1}{n}e\xrightarrow[]{\eta\theta} 0_{F}$, according to (ii). In either case it follows from Theorem \ref{locs} that $\eta\theta$ is linear.
\end{proof}

We will call $f\in F$ an \emph{$\eta$-almost eventual upper bound} for a net $\left(f_{\alpha}\right)_{\alpha\in A}\subset F$ if $\left(f_{\alpha}-f\right)^{+}\xrightarrow[]{\eta}0_{F}$. Modifying a notation from \cite{erz}, we will denote the collection of all $\eta$-almost eventual upper bounds of this net by $\left(f_{\alpha}\right)_{\alpha\in A}^{\nearrow,\eta}$. \emph{$\eta$-almost eventual lower bounds} are defined analogously and their collection is denoted by $\left(f_{\alpha}\right)_{\alpha\in A}^{\searrow,\eta}$. Clearly, $f_{\alpha}\xrightarrow[]{\eta}f$ iff $f\in \left(f_{\alpha}\right)_{\alpha\in A}^{\nearrow,\eta}\cap \left(f_{\alpha}\right)_{\alpha\in A}^{\searrow,\eta}$, and for a positive net $f_{\alpha}\xrightarrow[]{\eta}0_{F}$ iff $0_{F}\in\left(f_{\alpha}\right)_{\alpha\in A}^{\nearrow,\eta}$. If $\eta$ is a discrete convergence, then $f$ is an $\eta$-almost eventual upper (lower) bound iff it is an eventual upper (lower) bound.

\begin{proposition}\label{aeub}
Let $\eta$ and $\theta$ be locally solid additive convergences on $F$ and let $\left(f_{\alpha}\right)_{\alpha\in A}\subset F$. Then:
\item[(i)] $\left(f_{\alpha}\right)_{\alpha\in A}^{\nearrow,\eta}$ is a co-ideal in $F$.
\item[(ii)] $\left(f_{\alpha}\right)_{\alpha\in A}^{\nearrow,\eta\theta}=\overline{\left(f_{\alpha}\right)_{\alpha\in A}^{\nearrow,\eta}}^{1}_{\theta}$.
\item[(iii)] If $\eta$ is Hausdorff then $\left(f_{\alpha}\right)_{\alpha\in A}^{\searrow,\eta}\le\left(f_{\alpha}\right)_{\alpha\in A}^{\nearrow,\eta}$. In particular, if $\left(f_{\alpha}\right)_{\alpha\in A}\subset F_{+}$, then $\left(f_{\alpha}\right)_{\alpha\in A}^{\nearrow,\eta}\subset F_{+}$.
\end{proposition}
\begin{proof}
Let $G:=\left(f_{\alpha}\right)_{\alpha\in A}^{\nearrow,\eta}$.

(i): It is easy to see that if $h\ge g\in G$, then $h\in G$. If $g,h\in G$, then $\left(f_{\alpha}-g\wedge h\right)^{+}=\left(f_{\alpha}-g\right)^{+}\vee \left(f_{\alpha}-h\right)^{+}\xrightarrow[]{\eta} 0_{F}$, hence, $g\wedge h\in G$.\medskip

(ii): Assume that $g\in \overline{G}^{1}_{\theta}$, so that there is a net $\left(g_{\beta}\right)_{\beta\in B}\subset G$ which $\theta$-converges to $g$. We have $\left(\left(f_{\alpha}-g\right)^{+}-\left(g_{\beta}-g\right)^{+}\right)^{+}\le \left(f_{\alpha}-g_{\beta}\right)^{+}\xrightarrow[]{\eta}0_{F}$, for every $\beta$, hence $\left(f_{\alpha}-g\right)^{+}\xrightarrow[]{\eta\theta}0_{F}$ with a controlling net $\left(\left(g_{\beta}-g\right)^{+}\right)_{\beta\in B}$.

Conversely, assume that $\left(f_{\alpha}-g\right)^{+}\xrightarrow[]{\eta\theta}0_{F}$ with a controlling net $\left(h_{\beta}\right)_{\beta\in B}$. For every $\beta$ we have $\left(f_{\alpha}-g-h_{\beta}\right)^{+}=\left(\left(f_{\alpha}-g\right)^{+}-h_{\beta}\right)^{+}\xrightarrow[]{\eta}0_{F}$, hence $g+h_{\beta}\in G$. Since $g+h_{\beta}\xrightarrow[]{\theta}g$, we conclude that $g\in \overline{G}^{1}_{\theta}$.\medskip

(iii): If $f\in\left(f_{\alpha}\right)_{\alpha\in A}^{\searrow,\eta}$ and $g\in\left(f_{\alpha}\right)_{\alpha\in A}^{\nearrow,\eta}$, then $\left(f-g\right)^{+}\le \left(f-f_{\alpha}\right)^{+}+\left(f_{\alpha}-g\right)^{+}\xrightarrow[]{\eta}0_{F}$, hence $\left(f-g\right)^{+}=0_{F}$.
\end{proof}

Let us consider some alternative ways of describing the $\eta\theta$-convergence. First, note that the controlling net can be replaced with any of its quasi-subnets. Moreover, let us show that this net can be assumed to be a co-ideal.

\begin{proposition}\label{prod}Let $\eta$ and $\theta$ be locally solid additive convergences on $F$. For a net $\left(f_{\alpha}\right)_{\alpha\in A}\subset F_{+}$ TFAE:
\item[(i)] $f_{\alpha}\xrightarrow[]{\eta\theta} 0_{F}$;
\item[(ii)] There are nets $\left(h_{\alpha\beta}\right)_{\alpha\in A,~\beta\in B}$ and $\left(g_{\beta}\right)_{\beta\in B}$ such that $f_{\alpha}\le h_{\alpha\beta}\xrightarrow[\alpha]{\eta}g_{\beta}\xrightarrow[]{\theta}0_{F}$;
\item[(iii)] There is a $\theta$-null co-ideal $G\subset F_{+}$ such that $\left(f_{\alpha}-g\right)^{+}\xrightarrow[]{\eta} 0_{F}$, for every $g\in G$;
\item[(iv)] $0_{F}\in\overline{\left(f_{\alpha}\right)_{\alpha\in A}^{\nearrow,\eta}}^{1}_{\theta}$.\medskip

Moreover, for a filter $\Fo$ on $F$ we have $\Fo\xrightarrow[]{\eta\theta} 0_{F}$ iff there is a (decreasing) $\theta$-null net $\left(g_{\beta}\right)_{\beta\in B}\subset F_{+}$ such that $\left(\left|\Fo\right|-g_{\beta}\right)^{+}\xrightarrow[]{\eta} 0_{F}$, for every $\beta$.
\end{proposition}
\begin{proof}
(i)$\Rightarrow$(ii): If $f_{\alpha}\xrightarrow[]{\eta\theta} 0_{F}$, then $f_{\alpha}\le g_{\beta}+\left(f_{\alpha}-g_{\beta}\right)^{+}\xrightarrow[\alpha]{\eta}g_{\beta}\xrightarrow[]{\theta}0_{F}$, where $\left(g_{\beta}\right)_{\beta\in B}\subset F_{+}$ is a controlling net.

(ii)$\Rightarrow$(i): If $f_{\alpha}\le h_{\alpha\beta}\xrightarrow[\alpha]{\eta}g_{\beta}\xrightarrow[]{\theta}0_{F}$, then $\left(f_{\alpha}-g_{\beta}\right)^{+}\le\left(h_{\alpha\beta}-g_{\beta}\right)^{+} \xrightarrow[\alpha]{\eta}0_{F}$.\medskip

(i)$\Rightarrow$(iv) follows from part (ii) of Proposition \ref{aeub}, and (iii)$\Rightarrow$(i) is trivial.

(iv)$\Rightarrow$(iii): Let $G:=\left(f_{\alpha}\right)_{\alpha\in A}^{\nearrow,\eta}\cap F_{+}$, which is a co-ideal in $F_{+}$. From the co-ideal version of Proposition \ref{adh} it follows that $0_{F}\in \overline{G}^{1}_{\eta}$. According to Corollary \ref{coid}, this yields $G\xrightarrow[]{\theta} 0_{F}$.\medskip

We are left with proving the filter version. Let $\Fo=\left[f_{\alpha},~\alpha\in A\right]$, for some net $\left(f_{\alpha}\right)_{\alpha\in A}\subset F$. Then, $\Fo\xrightarrow[]{\eta\theta} 0_{F}$ $\Leftrightarrow$ $f_{\alpha}\xrightarrow[]{\eta\theta} 0_{F}$ $\Leftrightarrow$ $\left|f_{\alpha}\right|\xrightarrow[]{\eta\theta} 0_{F}$ $\Leftrightarrow$ there is a (decreasing) $\theta$-null net $\left(g_{\beta}\right)_{\beta\in B}\subset F_{+}$ such that $\left(\left|f_{\alpha}\right|-g_{\beta}\right)^{+}\xrightarrow[]{\eta} 0_{F}$, for every $\beta$. It is left to notice that $\left(\left|\Fo\right|-g_{\beta}\right)^{+}=\left[\left(\left|f_{\alpha}\right|-g_{\beta}\right)^{+},~\alpha\in A\right]$, for every $\beta$.
\end{proof}

It now follows that the only information passed down from $\theta$ are the decreasing $\theta$-null nets, or even more precisely, $\theta$-null co-ideals.

\begin{corollary}\label{mon}
Let $\eta,\eta_{1},\eta_{2},\theta_{1},\theta_{2}$ be locally solid additive convergences on $F$. Then:
\item[(i)] If $\eta_{1}\ge \eta_{2}$ and $\theta_{1}\ge\theta_{2}$, then $\eta_{1}\theta_{1}\ge \eta_{2}\theta_{2}$.
\item[(ii)] If $\theta_{1},\theta_{2}$ have the same null co-ideals in $F_{+}$ then $\eta\theta_{1}=\eta\theta_{2}$.
\end{corollary}

Using the condition (iv) in Proposition \ref{prod} and the fact that the adherence with respect to the supremum of a family of convergences is the intersection of the adherences with respect to each member of the family, we arrive at the following result.

\begin{proposition}
Let $\eta$ be a locally solid additive convergences on $F$, and let $\Theta$ be a collection of locally solid additive convergences on $F$. Then, $\eta\bigvee\Theta=\bigvee\limits_{\theta\in\Theta}\eta\theta$.
\end{proposition}

\begin{example}[Essentially \cite{fremlin}]\label{ru2}
We will show that if in $F$ every countable disjoint set is order bounded, then $\mathrm{ru}^{2}=\sigma\mathrm{o}$. On one hand, $\mathrm{ru}\ge\sigma\mathrm{o}$, hence $\mathrm{ru}^{2}\ge\sigma\mathrm{o}^{2}=\sigma\mathrm{o}$ (the last equality will be proven in Example \ref{sigmao}). On the other hand, assume that $0_{F}\le f_{\alpha}\xrightarrow[]{\sigma\mathrm{o}} 0_{F}$, so that there is a decreasing sequence $\left(e_{n}\right)_{n\in\N}\subset F$ with $\bigwedge\limits_{n\in\N}e_{n}=0_{F}$ and such that for every $n\in\N$ there is $\alpha_{n}$ with $f_{\alpha}\le e_{n}$, as soon as $\alpha\ge\alpha_{n}$. By \cite[Lemma 7.7]{ab0} for every $m\in\N$ there is $h_{m}\in F_{+}$ such that $e_{n}\le 2^{-n}h_{m}+\frac{1}{m}e_{1}$, for every $n\in\N$. Clearly, $\frac{1}{m}e_{1}\xrightarrow[m]{\mathrm{ru}} 0_{F}$; we claim that it is a controlling net for $f_{\alpha}\xrightarrow[]{\mathrm{ru}^{2}} 0_{F}$. Indeed, if $\alpha\ge\alpha_{n}$ we have that $\left(f_{\alpha}-\frac{1}{m}e_{1}\right)^{+}\le \left(e_{n}-\frac{1}{m}e_{1}\right)^{+}\le 2^{-n}h_{m}$, and so $\left(f_{\alpha}-\frac{1}{m}e_{1}\right)^{+}\xrightarrow[]{\mathrm{ru}} 0_{F}$, for every $m\in\N$.
\qed\end{example}

There can be various convergences which share decreasing null nets. In fact, we have the following general result.

\begin{theorem}\label{intd}
Let $\eta_{0}$ be the discrete convergence on $F$ and let $\theta$ be a locally solid additive convergence on $F$. Then:
\item[(i)] A net $\left(f_{\alpha}\right)_{\alpha\in A}\subset F_{+}$ is $\eta_{0}\theta$-null iff there is a (decreasing) $\theta$-null net $\left(g_{\beta}\right)_{\beta\in B}\subset F_{+}$ such that for every $\beta$ there is $\alpha_{0}$ such that $f_{\alpha}\le g_{\beta}$, for all $\alpha\ge\alpha_{0}$.
\item[(ii)] $\eta_{0}\theta$ and $\theta$ have the same decreasing null nets. In fact, $\eta_{0}\theta$ is the strongest locally solid additive convergence on $F$ in which every decreasing $\theta$-null net is null.
\item[(iii)] If $\theta$ is linear and completely metrizable, then $\eta_{0}\theta=\mathrm{ru}$.
\end{theorem}
\begin{proof}
(i) follows from the definition. The second claim in (ii) follows from (i) and \cite[Lemma 3.7]{ectv}; in particular $\theta\le\eta_{0}\theta$. On the other hand, according to part (ii) of Proposition \ref{produc} every decreasing $\theta$-null net is $\eta_{0}\theta$-null. This proves the first claim.\medskip

(iii): Since $\mathrm{ru}$ is the strongest linear locally solid convergence on $F$ (see \cite[Proposition 5.1]{ectv}), we only need to show that every decreasing $\theta$-null net is $\mathrm{ru}$-null. It is well-known that if $\theta$ is metrizable, it is metrizable by a Riesz pseudo-norm. Let $\rho$ be such and assume that a decreasing net $\left(f_{\alpha}\right)_{\alpha\in A}\subset F_{+}$ is $\theta$-null. There are $\left(\alpha_{n}\right)_{n\in\N}\subset A$ such that $\rho\left(f_{\alpha_{n}}\right)\le\frac{1}{n2^{n}}$, so that $\rho\left(nf_{\alpha_{n}}\right)\le\frac{1}{2^{n}}$, for every $n\in\N$. Let $e:=\sum\limits_{n\in\N}nf_{n}$, which exists due to completeness of $\rho$. It follows that $f_{\alpha_{n}}\le\frac{1}{n}e$, hence $f_{\alpha_{n}}\xrightarrow[]{\mathrm{ru}} 0_{F}$. We can now use Lemma \ref{mono} and conclude that $f_{\alpha}\xrightarrow[]{\mathrm{ru}} 0_{F}$.
\end{proof}

\begin{example}
Combining Example \ref{aw} with part (iii) of Theorem \ref{intd} yields that on a Banach lattice $\eta_{0}\mathrm{aw}=\mathrm{ru}$.
\qed\end{example}

\begin{remark}\label{vertical}
It is not hard to show that $\eta_{0}\left(\eta_{0}\theta\right)=\eta_{0}\theta$, for any locally solid additive convergence $\theta$ on $F$. Hence, for $\theta$ we have that $\eta_{0}\theta=\theta$ iff $\theta=\eta_{0}\zeta$, for some locally solid additive convergence $\zeta$ on $F$. We will call convergences with this property \emph{vertical}; they will be investigated elsewhere. Order, $\sigma$-order and relative uniform convergences are vertical. Note that for vertical $\theta$ we have that $\theta\ge\eta\theta$, for any $\eta$.
\qed\end{remark}

\begin{example}
Unsurprisingly, composition is not commutative. Let $\eta$ be the norm convergence on a Banach lattice without a strong unit, and let $\theta$ be relative uniform convergence, which is strictly stronger than $\eta$. Since they have the same decreasing null nets, it follows that $\eta\theta=\eta\eta=\eta$ and $\theta\eta=\theta\theta=\theta$ (see \cite[Examples 2.5 and 2.6]{erz}).
\qed\end{example}

\subsection{Associativity of composition}

We will now study the operation of adherence with respect to the composition of convergences.

\begin{theorem}\label{product}
Let $\eta,\theta$ be locally solid additive convergences on $F$.
\item[(i)] If $H\subset F_{+}$ is positively solid, then $\overline{H}^{1}_{\eta\theta}\subset\overline{\overline{H}^{1}_{\eta}}^{1}_{\theta}$ (adherences in $F_{+}$). If $H$ is a lattice-ideal, then $\overline{H}^{1}_{\eta\theta}=\overline{\overline{H}^{1}_{\eta}}^{1}_{\theta}$.
\item[(ii)] If $H\subset F$ is solid, then $\overline{H}^{1}_{\eta\theta}\subset\overline{\overline{H}^{1}_{\eta}}^{1}_{\theta}$. If $H$ is an ideal, then $\overline{H}^{1}_{\eta\theta}=\overline{\overline{H}^{1}_{\eta}}^{1}_{\theta}$.
\end{theorem}
\begin{proof}
(i): Let $f\in \overline{H}^{1}_{\eta\theta}$. By Lemma \ref{psola} there is a net $\left(f_{\alpha}\right)_{\alpha\in A}\subset H\cap\left[0_{F},f\right]$ which $\eta\theta$-converges to $f$. Then, there is a $\theta$-null co-ideal $G\subset F_{+}$ such that  $\left(\left(f-g\right)^{+}-f_{\alpha}\right)^{+}=\left(f-f_{\alpha}-g\right)^{+}\xrightarrow[]{\eta} 0_{F}$, for every $g\in G$. According to Lemma \ref{psola} $\left(f-g\right)^{+}\in \overline{H}^{1}_{\eta}$, for every $g\in G$, and as $G$ is $\theta$-null, we conclude that $f\in\overline{\overline{H}^{1}_{\eta}}^{1}_{\theta}$.\medskip

We now assume that $H$ is a lattice ideal and $f\in\overline{\overline{H}^{1}_{\eta}}^{1}_{\theta}$, so that there is  $\left(f_{\beta}\right)_{\beta\in B}\subset \left[0_{F},f\right]\cap \overline{H}^{1}_{\eta}$ which $\theta$-converges to $f$. Then, $g_{\beta}\xrightarrow[]{\theta} 0_{F}$, where $g_{\beta}=f-f_{\beta}$, for every $\beta$. For every $\beta$ there is $\left(e_{\gamma}^{\beta}\right)_{\gamma\in \Gamma_{\beta}}\subset \left[0_{F},f_{\beta}\right]\cap H\subset \left[0_{F},f\right]\cap H$ which $\eta$-converge to $f_{\beta}=f-g_{\beta}$. Hence, the net $\left(\left(f-u-g_{\beta}\right)^{+}\right)_{u\in \left[0_{F},f\right]\cap H}$ is positive, decreasing and contains $\left(\left(f_{\beta}-e_{\gamma}^{\beta}\right)^{+}\right)_{\gamma\in\Gamma_{\beta}}$, which is $\eta$-null, therefore is itself $\eta$-null, according to Lemma \ref{mono}. Thus, $f-u\xrightarrow[{u\in \left[0_{F},f\right]\cap H}]{\eta\theta}0_{F}$, from where $f\in \overline{H}^{1}_{\eta\theta}$.\medskip

(ii): Let $f\in \overline{H}^{1}_{\eta\theta}$. It is easy to see that then $\left|f\right|\in \overline{H_{+}}^{1}_{\eta\theta}\subset \overline{\overline{H_{+}}^{1}_{\eta}}^{1}_{\theta}\subset \overline{\overline{H}^{1}_{\eta}}^{1}_{\theta}$. As the latter set is solid, it follows that $f\in \overline{\overline{H}^{1}_{\eta}}^{1}_{\theta}$.\medskip

Assume that $H$ is an ideal and $f\in \overline{\overline{H}^{1}_{\eta}}^{1}_{\theta}$. One can show that then $\left|f\right|\in \overline{\overline{H_{+}}^{1}_{\eta}}^{1}_{\theta}=\overline{H_{+}}^{1}_{\eta\theta}\subset \overline{H}^{1}_{\eta\theta}$. Since the latter set is an ideal, we conclude that $f\in\overline{H}^{1}_{\eta\theta}$.
\end{proof}

Analogously, using Corollary \ref{coid} one can get the following result.

\begin{proposition}\label{coide}
If $G\subset F_{+}$ is a co-ideal, then $\overline{G}^{1}_{\eta\theta}=\overline{\overline{G}^{1}_{\eta}}^{1}_{\theta}$, and in particular $G\xrightarrow[]{\eta\theta}0_{F}$ iff $0_{F}\in\overline{\overline{G}^{1}_{\eta}}^{1}_{\theta}$.
\end{proposition}

\begin{corollary}\label{ass}
Composition is associative.
\end{corollary}
\begin{proof}
Let $\eta,\theta,\zeta$ be locally solid additive convergences on $F$. For a net $\left(f_{\alpha}\right)_{\alpha\in A}\subset F_{+}$ we have $$\left(f_{\alpha}\right)_{\alpha\in A}^{\nearrow,\eta\left(\theta\zeta\right)}=\overline{\left(f_{\alpha}\right)_{\alpha\in A}^{\nearrow,\eta}}^{1}_{\theta\zeta}=\overline{\overline{\left(f_{\alpha}\right)_{\alpha\in A}^{\nearrow,\eta}}^{1}_{\theta}}^{1}_{\zeta}=\overline{\left(f_{\alpha}\right)_{\alpha\in A}^{\nearrow,\eta\theta}}^{1}_{\zeta}=\left(f_{\alpha}\right)_{\alpha\in A}^{\nearrow,\left(\eta\theta\right)\zeta},$$
where the second equality follows from Proposition \ref{coide}, and all the rest -- from part (ii) of Proposition \ref{aeub}. Hence, $f_{\alpha}\xrightarrow[]{\eta\left(\theta\zeta\right)}0_{F}$ $\Leftrightarrow$ $0_{F}\in \left(f_{\alpha}\right)_{\alpha\in A}^{\nearrow,\eta\left(\theta\zeta\right)}=\left(f_{\alpha}\right)_{\alpha\in A}^{\nearrow,\left(\eta\theta\right)\zeta}$ $\Leftrightarrow$ $f_{\alpha}\xrightarrow[]{\left(\eta\theta\right)\zeta}0_{F}$.
\end{proof}

Let us describe how to take the composition of infinitely many convergences. The proof is a standard transfinite induction argument and relies on the fact that an infimum of a directed set of Hausdorff convergences is Hausdorff.

\begin{proposition}\label{infiniteco}
Let $n_{0}$ be an ordinal, and assume that for every successor ordinal $n\le n_{0}$ there is a locally solid additive convergence $\eta_{n}$ on $F$. Define $\prod\limits_{k=1}^{n}\eta_{k}$ inductively by $\prod\limits_{k=1}^{n}\eta_{k}:=\left(\prod\limits_{k=1}^{n-1}\eta_{k}\right)\eta_{n}$, if $n>1$ is a successor ordinal and $\prod\limits_{k=1}^{n}\eta_{k}:=\bigwedge\limits_{m<n}\prod\limits_{k=1}^{m}\eta_{k}$ otherwise. Then, $\prod\limits_{k=1}^{n_{0}}\eta_{k}$ is a locally solid additive convergence. It is linear if at least one of $\eta_{n}$ is and Hausdorff iff all $\eta_{n}$ are. Moreover, $\prod\limits_{k=1}^{n}\eta_{k}=\prod\limits_{k=1}^{m}\eta_{k} \prod\limits_{k=m+1}^{n}\eta_{k}$, for any $m<n$.
\end{proposition}

\subsection{Composition vs. operations with convergences}

In this subsection we obtain a series of similar results about the way various operations with convergences ``commute'' with composition. We start with the extension a la Example \ref{propadd}.

\begin{proposition}\label{propa}
If $E\subset F$ is an ideal, and $\eta,\theta$ are locally solid additive convergences on $E$, then $\eta_{F}\theta_{F}=\left(\eta\theta\right)_{F}$.
\end{proposition}
\begin{proof}
If $f_{\alpha}\xrightarrow[]{\eta_{F}\theta_{F}} 0_{F}$ then there is a $\theta$-null net $\left(g_{\beta}\right)_{\beta\in B}\subset E_{+}$ such that $\left(f_{\alpha}-g_{\beta}\right)^{+}\xrightarrow[]{\eta} 0_{F}$, for every $\beta$; in particular $\left(f_{\alpha}\right)_{\alpha\in A}$ is eventually in $E$, and so it is $\left(\eta\theta\right)_{F}$-null. The converse implication is proven similarly.
\end{proof}

We now look at how composition interacts with the unbounded modification.

\begin{proposition}\label{unbo}
If $E\subset F$ is an ideal and $\eta,\theta$ are locally solid additive convergences on $F$, then $\mathrm{u}_{E}\eta \mathrm{u}_{E}\theta\ge \mathrm{u}_{E}\left(\eta\theta\right)$.
\end{proposition}
\begin{proof}
Assume that $\left(f_{\alpha}\right)_{\alpha\in A}\subset F_{+}$ is $\mathrm{u}_{E}\eta \mathrm{u}_{E}\theta$-null. Then, by Proposition \ref{produc} there are nets $\left(h_{\alpha\beta}\right)_{\alpha\in A,~\beta\in B}$ and $\left(g_{\beta}\right)_{\beta\in B}$ in $F_{+}$ such that $0_{F}\le f_{\alpha}\le h_{\alpha,\beta}\xrightarrow[\alpha]{\mathrm{u}_{E}\eta} g_{\beta}\xrightarrow[\beta]{\mathrm{u}_{E}\theta} 0_{F}$. For every $e\in E_{+}$ we have $f_{\alpha}\wedge e\le h_{\alpha,\beta}\wedge e\xrightarrow[\alpha]{\eta} g_{\beta}\wedge e\xrightarrow[\beta]{\theta} 0_{F}$. Using Proposition \ref{produc} again we obtain $e\wedge f_{\alpha}\xrightarrow[]{\eta\theta} 0_{F}$. As $e$ was arbitrary, we conclude that $f_{\alpha}\xrightarrow[]{\mathrm{u}_{E}\left(\eta\theta\right)} 0_{F}$.
\end{proof}

\begin{remark}
Recall that $\theta$ and $\mathrm{u}\theta$ have the same decreasing null nets. Hence, according to Corollary \ref{mon}, $\mathrm{u}\eta \mathrm{u}\theta=\left(\mathrm{u}\eta\right)\theta$. It is not clear whether $\mathrm{u}_{E}\eta \mathrm{u}_{E}\theta=\left(\mathrm{u}_{E}\eta\right)\theta$, for any ideal $E\subset F$.
\qed\end{remark}

Next, we consider what happens with composition under a pull-back. If $T:F\to E$ is a positive operator, and $\eta$ is a locally solid additive convergence on $E$, define the convergence $\eta_{T}$ on $F$ by $0_{F}\le f_{\alpha}\xrightarrow[]{\eta_{T}}0_{F}$ if $Tf_{\alpha}\xrightarrow[]{\eta}0_{E}$ (this determines an additive locally solid convergence on $F$ by Theorem \ref{locs}).

\begin{proposition}\label{pullb}
If $T:F\to E$ is a positive operator, and $\eta,\theta$ are locally solid additive convergences on $E$, then $\eta_{T}\theta_{T}\ge\left(\eta\theta\right)_{T}$. If $T$ is a surjective homomorphism, then $\eta_{T}\theta_{T}=\left(\eta\theta\right)_{T}$.
\end{proposition}
\begin{proof}
Assume that $0_{F}\le f_{\alpha}\xrightarrow[]{\eta_{T}\theta_{T}} 0_{F}$ so that there is a $\theta_{T}$-null net $\left(g_{\beta}\right)_{\beta\in B}\subset F_{+}$ such that $\left(f_{\alpha}-g_{\beta}\right)^{+}\xrightarrow[]{\eta_{T}} 0_{F}$, for every $\beta$. Then, $Tg_{\beta}\xrightarrow[]{\eta}0_{F}$ and $0_{F}\le \left(Tf_{\alpha}-Tg_{\beta}\right)^{+}\le T\left(f_{\alpha}-g_{\beta}\right)^{+}\xrightarrow[]{\eta_{T}} 0_{F}$, for every $\beta$. This means that $\left(Tf_{\alpha}\right)_{\alpha\in A}$ is $\eta\theta$-null with the controlling net $\left(Tg_{\beta}\right)_{\beta\in B}$.\medskip

Assume that $T$ is a surjective homomorphism and $0_{F}\le f_{\alpha}\xrightarrow[]{\left(\eta\theta\right)_{T}} 0_{F}$, so that there is a $\theta$-null net $\left(e_{\beta}\right)_{\beta\in B}\subset E_{+}$ such that $\left(Tf_{\alpha}-e_{\beta}\right)^{+}\xrightarrow[]{\eta} 0_{F}$, for every $\beta$. Since $T$ is a surjective homomorphism, there is $g_{\beta}\in F_{+}$ such that $Tg_{\beta}=e_{\beta}$, for every $\beta$. It follows that $g_{\beta}\xrightarrow[]{\theta_{T}}0_{F}$, and $T\left(f_{\alpha}-g_{\beta}\right)^{+}=\left(Tf_{\alpha}-e_{\beta}\right)^{+}\xrightarrow[]{\eta} 0_{F}$, for every $\beta$. Thus, $ f_{\alpha}\xrightarrow[]{\eta_{T}\theta_{T}} 0_{F}$.
\end{proof}

The final operation we consider in this section is the Choquet modification (see \cite[Section 9]{ectv}).

\begin{proposition}\label{choq}
If $\eta,\theta$ are locally solid additive convergences on $F$, then $\mathrm{c}\eta \mathrm{c}\theta=\left(\mathrm{c}\eta\right)\theta\ge \mathrm{c}\left(\eta\theta\right)$.
\end{proposition}
\begin{proof}
First, recall that $\theta$ agrees with $\mathrm{c}\theta$ on monotone nets (see \cite[Proposition 9.7]{ectv}). Hence, according to Corollary \ref{mon}, $\mathrm{c}\eta \mathrm{c}\theta=\left(\mathrm{c}\eta\right)\theta$.\medskip

Assume that $\Fo$ is a $\left(\mathrm{c}\eta\right)\theta$-null filter on $F$, so that there is a $\theta$-null net $\left(g_{\beta}\right)_{\beta\in B}\subset F_{+}$ such that $\left(\left|\Fo\right|-g_{\beta}\right)^{+}\xrightarrow[]{\mathrm{c}\eta} 0_{F}$, for every $\beta$. Let $\mathcal{U}\supset\Fo$ be an unltrafilter. For every $\beta$, we have that $\left(\left|\mathcal{U}\right|-g_{\beta}\right)^{+}$ is an ultrafilter on $F$ that contains $\left(\left|\Fo\right|-g_{\beta}\right)^{+}$, and so is $\eta$-null. As $\beta$ is arbitrary, according to Proposition \ref{prod}, $\mathcal{U}\xrightarrow[]{\eta\theta}0_{F}$, and since $\mathcal{U}$ is arbitrary, we conclude that $\Fo\xrightarrow[]{\mathrm{c}\left(\eta\theta\right)}0_{F}$.
\end{proof}

The last three results motivate the following questions.

\begin{question}
Do we always have $\mathrm{c}\eta \mathrm{c}\theta=\mathrm{c}\left(\eta\theta\right)$, $\eta_{T}\theta_{T}=\left(\eta\theta\right)_{T}$ and $\mathrm{u}_{E}\eta \mathrm{u}_{E}\theta=\mathrm{u}_{E}\left(\eta\theta\right)$?
\end{question}

\subsection{Idempotent convergences}\label{idem}

The remaining part of the section will be devoted to the \emph{idempotent} convergences, i.e. locally solid additive convergences $\eta$ on $F$ such that $\eta^{2}=\eta$.

\begin{example}
It was observed in \cite[Examples 2.5 and 2.6, and Proposition 4.5]{erz} that order convergence along with any locally solid additive topology is idempotent, while relative uniform convergence is idempotent under the additional assumption that $F$ has $\sigma$-property (every sequence is contained in a principal ideal).\medskip

Relative uniform convergence can be idempotent even without $\sigma$-property. Indeed, let $F=c_{00}$ which clearly does not have the $\sigma$-property. Since $F$ is Archimedean, order convergence on $F$ is weaker than relative uniform convergence. On the other hand, every order convergent net is eventually order bounded, hence finitely dimensional, and so order convergence coincides with the finite-dimensional convergence on $F$. As the latter is the strongest linear convergence on $F$ (see \cite[Proposition 7.12]{ectv}), all three coincide. Since order convergence is always idempotent, it follows that relative uniform convergence is idempotent on $F$.
\qed\end{example}

\begin{question}
Is it true that if $X$ is a Tychonoff space such that $\mathrm{ru}$ is idempotent on $\Co\left(X\right)$, then the latter has $\sigma$-property? Same question for $\Co^{\8}\left(X\right)$, for $X$ extremally disconnected.
\end{question}

\begin{example}\label{sigmao}
Let us show that $\sigma$-order convergence is also idempotent. Assume that $\left(f_{\alpha}\right)_{\alpha\in A}\subset F_{+}$ is such that there is a decreasing $\mathrm{\sigma o}$-null net $\left(g_{\beta}\right)_{\beta\in B}$ such that $\left(f_{\alpha}-g_{\beta}\right)^{+}\xrightarrow[]{\mathrm{\sigma o}}0_{F}$, for every $\beta$. Since $\left(g_{\beta}\right)_{\beta\in B}$ is decreasing and $\mathrm{\sigma o}$-null, it is not hard to show that there are $\left(\beta_{n}\right)_{n\in \N}$ such that $\bigwedge\limits_{n\in\N}g_{\beta_{n}}=0_{F}$. For every $n\in\N$ let $H_{n}$ be a countable set which witnesses $\left(f_{\alpha}-g_{\beta_{n}}\right)^{+}\xrightarrow[]{\mathrm{\sigma o}}0_{F}$. It is left to observe that $H:=\bigcup\limits_{n\in\N}\left(g_{\beta_{n}}+H_{n}\right)$ witnesses $f_{\alpha}\xrightarrow[]{\mathrm{\sigma o}}0_{F}$.
\qed\end{example}\smallskip

\begin{example}\label{idempotent}
The unbounded modification of an idempotent convergence is idempotent. Let $\eta$ be an idempotent convergence on $F$ and let $E\subset F$ be an ideal. Then, $\mathrm{u}_{E}\eta\ge\mathrm{u}_{E}\eta\mathrm{u}_{E}\eta\ge \mathrm{u}_{E}\left(\eta^{2}\right)=\mathrm{u}_{E}\eta$,
where the first inequality follows from part (i) of Proposition \ref{produc}, the second one -- from Proposition \ref{unbo}, whereas the last equality results from idempotency of $\eta$.\medskip

Analogously, using propositions \ref{pullb} and \ref{choq} one can show that the Choquet modification and any pull-back of an idempotent convergence are idempotent. In particular, a restriction of an idempotent convergence to a sublattice of $F$ is idempotent.
\qed\end{example}\smallskip

\begin{example}\label{supid}
It can be easily checked that the supremum of any collection of idempotent convergences is idempotent.
\qed\end{example}\smallskip

\begin{example}\label{idemod}
Any locally solid additive convergence can be ``idempotentified''. Indeed, by Example \ref{supid} the supremum of all idempotent convergences weaker than a given convergence $\eta$ is idempotent. Let us consider a more explicit way to construct it.

Let $n_{0}$ be the first ordinal of a cardinality larger than that of the set of all convergences on $F$. In the notation of Proposition \ref{infiniteco}, let $\eta_{n}=\eta$, for every successor ordinal $n<n_{0}$. Then, the decreasing net $\left(\eta^{n}\right)_{n\le n_{0}}$ of convergences has to stabilize by the cardinality reasons. It is then easy to see by Proposition \ref{infiniteco} that $\eta^{n_{0}}$ is idempotent.

The class of vector lattices with $\mathrm{ru}^{\omega_{1}}=\sigma\mathrm{o}$ was investigated in \cite{fremlin}. That paper also contains examples from this class which show that any countable ordinal might be needed to idempotify $\mathrm{ru}$. In particular, it was observed there that for $F=\R^{\left[0,1\right]}$, we have $\mathrm{ru}>\sigma\mathrm{o}=\mathrm{ru}^{2}$ (see also Example \ref{ru2}), which is the first explicit example of a non-idempotent locally solid convergence.
\qed\end{example}\smallskip

\begin{example}
Let $X$ be a convergence space. Recall that a net $\left(f_{\alpha}\right)_{\alpha\in A}\subset \Co\left(X\right)$ is \emph{continuously} convergent to $f\in\Co\left(X\right)$ if for every $x\in X$ and any net $\left(x_{\gamma}\right)_{\gamma\in\Gamma}\subset X$ convergent to $x$, we have $f_{\alpha}\left(x_{\gamma}\right)\xrightarrow[\alpha,\gamma]{}f\left(x\right)$. Let us show that this convergence is idempotent. Assume that $\0\le g_{\beta}\xrightarrow[]{\mathrm{c}}\0$ and $\left(f_{\alpha}\right)_{\alpha\in A}\subset \Co\left(X\right)_{+}$ is such that $\left(f_{\alpha}-g_{\beta}\right)^{+}\xrightarrow[]{\mathrm{c}}\0$, for every $\beta$. Fix $x$, $x_{\gamma}\to x$ and $\varepsilon>0$. We have that $g_{\beta}\left(x_{\gamma}\right)\xrightarrow[]{}0$, and so there are $\beta_{0},\gamma_{0}$ such that $g_{\beta}\left(x_{\gamma}\right)\le\frac{\varepsilon}{2}$, for all $\beta\ge\beta_{0}$ and $\gamma\ge\gamma_{0}$. As $\left(f_{\alpha}-g_{\beta_{0}}\right)^{+}\xrightarrow[]{c}\0$, there are $\alpha_{0}$ and $\gamma_{1}\ge\gamma_{0}$ such that $\left(f_{\alpha}\left(x_{\gamma}\right)-g_{\beta_{0}}\left(x_{\gamma}\right)\right)^{+}\le \frac{\varepsilon}{2}$, for all $\alpha\ge\alpha_{0}$ and $\gamma\ge\gamma_{1}$. Then, $f_{\alpha}\left(x_{\gamma}\right)\le \left(f_{\alpha}\left(x_{\gamma}\right)-g_{\beta_{0}}\left(x_{\gamma}\right)\right)^{+}+g_{\beta_{0}}\left(x_{\gamma}\right)\le \frac{\varepsilon}{2}+\frac{\varepsilon}{2}\le\varepsilon$, for all $\alpha\ge\alpha_{0}$ and $\gamma\ge\gamma_{1}$. Since $\varepsilon$ was chosen arbitrarily, it follows that $f_{\alpha}\left(x_{\gamma}\right)\xrightarrow[\alpha,\gamma]{}0$. Due to arbitrariness of $\left(x_{\gamma}\right)_{\gamma\in\Gamma}$, we conclude that $f_{\alpha}\xrightarrow[]{\mathrm{c}}\0$.
\qed\end{example}\smallskip

\begin{remark}
Let $X$ be a Tychonoff topological space. One can show that $\0\le f_{\alpha}\xrightarrow[]{\mathrm{c}}\0$ iff for every $x\in X$ and for every $\varepsilon>0$ there are an open neighborhood $U$ of $x$ and $\alpha_{0}$ such that $\left.f_{\alpha}\right|_{U}\le\varepsilon$, for every $\alpha\ge\alpha_{0}$. Interchanging the middle quantifiers we arrive at \emph{local uniform} convergence: $\0\le f_{\alpha}\xrightarrow[]{\mathrm{lu}}\0$ if for every $x\in X$ there is an open neighborhood $U$ of $x$ such that for every $\varepsilon>0$ there is $\alpha_{0}$ such that $\left.f_{\alpha}\right|_{U}\le\varepsilon$, for every $\alpha\ge\alpha_{0}$. It follows immediately that local uniform convergence is a Hausdorff linear locally solid convergence. It would be interesting to characterize for which $X$'s, this convergence is idempotent. This is the case if $X$ is locally compact, since then $\mathrm{lu}$ coincides with the compact-open topology.
\qed\end{remark}

From the monotonicity of the product, it follows that if $\eta$ is idempotent and $\theta\ge\eta$, then $\eta\ge\eta\theta\ge\eta\eta=\eta$. If $\theta$ is vertical (see Example \ref{vertical}), and $\eta\ge\theta$, then $\theta\ge\eta\theta\ge\theta\theta=\theta$.

\begin{question}
Is the product of topological (idempotent) locally solid additive convergences topological (idempotent)?
\end{question}

It follows from Theorem \ref{product} that if $F$ is endowed with an idempotent convergence $\eta$, the adherence of every ideal in $F$ is closed (and so equal to its closure). Indeed, if $H\subset F$ is an ideal, then $\overline{\overline{H}^{1}_{\eta}}^{1}_{\eta}=\overline{H}^{1}_{\eta^{2}}=\overline{H}^{1}_{\eta}$. Similarly, the adherence in $F_{+}$ of any lattice-ideal of $F_{+}$ is also closed. Moreover, if $G\subset F_{+}$ is a co-ideal, then $G\to0_{F}$ iff $0_{F}\in\overline{G}$. It also follows from Proposition \ref{aeub}, that for any net $\left(f_{\alpha}\right)_{\alpha\in A}\subset F$, the set $\left(f_{\alpha}\right)_{\alpha\in A}^{\nearrow,\eta}$ is a closed co-ideal.\medskip

We conclude this section by showing that any additive idempotent convergence has a closed ``ideal of linearity''.

\begin{proposition}
If $F$ is endowed with an additive locally solid idempotent convergence, then $\left\{f\in F,~ \frac{1}{n}f\to 0_{F}\right\}$ is a closed ideal.
\end{proposition}
\begin{proof}
It is straightforward to verify that the set $E$ in question is an ideal. To show that it is closed, let $f\in\overline{E}^{1}_{+}\subset \overline{E_{+}}^{1}$, so that $\left[0_{F},f\right]\cap E\to f$. We have that $\frac{1}{n}f\le \frac{1}{n}e+f-e\xrightarrow[n\in\N]{}f-e\xrightarrow[{e\in \left[0_{F},f\right]\cap E}]{}0_{F}$, and since the convergence is idempotent, it follows that $\frac{1}{n}f\to 0_{F}$, hence $f\in E$.
\end{proof}

\section{Some applications}

\subsection{Existence of the weakest locally solid convergence}\label{weake}

In this subsection we will focus on minimal and weakest Hausdorff locally solid convergences, and close some gaps that were left open in \cite[Section 10]{ectv}. Throughout this section $F$ is an Archimedean vector lattice. The following simple fact will come very handy.

\begin{lemma}\label{weaker}
For every additive Hausdorff locally solid convergence on $F$ there is a weaker linear Hausdorff locally solid convergence.
\end{lemma}
\begin{proof}
If $\eta$ is an additive Hausdorff locally solid convergence on $F$, the product $\eta\mathrm{ru}$ is a convergence that satisfies the required properties due to Proposition \ref{produc}.
\end{proof}

\begin{question}
Find an explicit description of the strongest linear (locally solid) convergence weaker than a given additive (locally solid) one.
\end{question}

\begin{corollary}
\item[(i)] Every minimal additive Hausdorff locally solid convergence on $F$ is linear.
\item[(ii)] Every minimal linear Hausdorff locally solid convergence on $F$ is also a minimal additive Hausdorff locally solid convergence.
\item[(iii)] The weakest linear Hausdorff locally solid convergence on $F$ (if exists) is also the weakest additive Hausdorff locally solid convergence.
\end{corollary}
\begin{proof}
A minimal additive Hausdorff locally solid convergence has to be equal to the linear convergence from Lemma \ref{weaker}, and thus is itself linear. If $\eta$ is a minimal linear Hausdorff locally solid convergence, $\theta\le\eta$ is an additive Hausdorff locally solid convergence, and $\zeta\le\theta$ is the linear convergence from Lemma \ref{weaker}, it follows from minimality that $\zeta=\eta=\theta$. The last claim is proven similarly.
\end{proof}

It follows that the sets of minimal additive and linear Hausdorff locally solid convergences coincide. Therefore, we can speak of simply \emph{minimal Hausdorff locally solid convergences}. Similarly we can speak of \emph{the weakest Hausdorff locally solid convergence}.\medskip

It was proven in \cite[Proposition 10.4]{ectv} that if $F$ is an atomic Archimedean vector lattice then the coordinate-wise convergence is the weakest Hausdorff linear locally solid convergence. Before proving a kind of a converse to this statement, we need the following concept.\medskip

Call convergences $\eta,\theta$ on a set $X$ \emph{compatible} if the diagonal is closed in $\left(X\times X,\eta\times\theta\right)$, and \emph{incompatible} otherwise, i.e. there is a net which has two different limits with respect to $\eta,\theta$. In particular, $\eta$ is Hausdorff iff it is compatible with itself. Note that if $\eta,\theta$ are incompatible, and $\eta'\le\eta$, $\theta'\le\theta$, then $\eta',\theta'$ are incompatible. In particular there are no Hausdorff convergences weaker than both $\eta,\theta$.

\begin{theorem}\label{weakest}
TFAE:
\item[(i)] $F$ is atomic;
\item[(ii)] $\mathrm{uo}$ is a minimal Hausdorff locally solid convergence on $F$;
\item[(iii)] $\mathrm{cuo}$ is a minimal Hausdorff locally solid convergence on $F$;
\item[(iv)] There is a unique minimal Hausdorff locally solid convergence on $F$;
\item[(v)] $F$ admits the weakest Hausdorff locally solid convergence.
\end{theorem}
\begin{proof}
(i)$\Rightarrow$(ii): If $F$ is atomic, then the coordinate-wise convergence, which is the weakest Hausdorff locally solid convergence, coincides with $\mathrm{uo}$ (see \cite[Corollary 10.5]{ectv}).\medskip

(ii)$\Rightarrow$(iii) follows from $\mathrm{cuo}\le\mathrm{uo}$.\medskip

(iii)$\Rightarrow$(iv): It was proven in \cite[Proposition 10.9 and Theorem 10.13]{ectv} that every minimal Hausdorff locally solid convergence has to be Choquet, and weaker than $\mathrm{uo}$, hence weaker than $\mathrm{cuo}$. Since the latter is minimal, it is the only minimal Hausdorff locally solid convergence on $F$.\medskip

(iv)$\Rightarrow$(v) follows from the fact that the set of all Hausdorff locally solid (linear or additive) convergences satisfies the conditions of Zorn's lemma (downward).\medskip

(v)$\Rightarrow$(i): We will show that if $F$ is not atomic, there are two incompatible Hausdorff locally solid convergences on $F$. Since the weakest Hausdorff convergence would have to be weaker than both of them, it does not exist.\medskip

Let $e>0_{F}$ be in the atomless part of $F$. Let $J:F_{e}\to\Co\left(K\right)$ be an injective homomorphism with a dense range such that $Je=\1$. It follows from e.g. \cite[Lemma 4.2]{abt} that $K$ has no isolated points, and so according to \cite{hewitt} there is a dense $L\subset K$ such that $M:=K\backslash L$ is also dense. Let $\eta$ be a convergence on $\Co\left(K\right)$ defined by: $g_{\alpha}\xrightarrow{\eta}g$ if $g_{\alpha}\left(x\right)\to g\left(x\right)$, for every $x\in L$, i.e. it is the pointwise convergence over $L$. Let $\theta$ be the pointwise convergence over $M$ (also on $\Co\left(K\right)$). Both of these convergences are Hausdorff linear locally solid (because they are induced by separating collections of Riesz semi-norms).\medskip

We can now pull-back the convergences $\eta$ and $\theta$ along $J$ (which is injective) and obtain Hausdorff linear locally solid convergences $\eta_{J}$ and $\theta_{J}$ on $F_{e}$. Then, we extend these two convergences as explained before Example \ref{lru} to get Hausdorff locally solid additive convergences $\eta_{JF}$ and $\theta_{JF}$ on $F$. It is left to show that $\eta_{JF}$ and $\theta_{JF}$ are incompatible.\medskip

Let $A\subset L$ and $B\subset M$ be finite. Then, $A,B$ are closed disjoint sets, and so according to Sublattice Urysohn Lemma (see e.g. \cite[Proposition 2.1]{erz1}) there are $g_{A,B}\in JF_{e}\cap\left[\0,\1\right]$ such that $g_{A,B}\left(x\right)=0$, for every $x\in A$, and $g_{A,B}\left(x\right)=1$, for every $x\in B$. It follows that $g_{A,B}\xrightarrow[A,B]{\eta}\0$ and $g_{A,B}\xrightarrow[A,B]{\theta}\1$.

Let $f_{A,B}:=J^{-1}g_{A,B}\in F_{e}$, for finite $A\subset L$ and $B\subset M$. We have $Jf_{A,B}=g_{A,B}\xrightarrow[A,B]{\eta}\0$ and $J\left(e-f_{A,B}\right)=\1-g_{A,B}\xrightarrow[A,B]{\theta}\0$, hence $f_{A,B}\xrightarrow[A,B]{\eta_{J}}0_{F}$ and $f_{A,B}\xrightarrow[A,B]{\theta_{J}}e$. This implies that $f_{A,B}\xrightarrow[A,B]{\eta_{JF}}0_{F}$ and $f_{A,B}\xrightarrow[A,B]{\theta_{JF}}e$, which makes $\eta_{JF}$ and $\theta_{JF}$ incompatible.
\end{proof}

\begin{remark}
We answered \cite[Question 10.23]{ectv} in the negative. Along the way we also proved that it is not true that every Hausdorff Choquet unbounded idempotent order continuous locally solid linear convergence is a minimal Hausdorff locally solid convergence. Indeed, $\mathrm{cuo}$ is Choquet by definition, order continuous since both unbounded and Choquet modifications weaken the convergence, locally solid and unbounded by \cite[Propositions 9.7 and 9.8]{ectv}, idempotent by Example \ref{idempotent}, but is not minimal, unless $F$ is atomic.
\qed\end{remark}

\begin{question}
Is $\mathrm{cuo}$ the supremum of all minimal Hausdorff locally solid convergences on $F$?
\end{question}

\subsection{Applications to unbounded modification}\label{unb}

Throughout this subsection $F$ is an Archimedean vector lattice. Recall that the unbounded modification $\mathrm{u}_{A}\eta$ of a locally solid convergence $\eta$ on $F$ by $A\subset F$ is defined by $0_{F}\le f_{\alpha}\xrightarrow[]{\mathrm{u}_{A}\eta}0_{F}$ if $\left|a\right|\wedge f_{\alpha}\xrightarrow[]{\eta}0_{F}$, for every $a\in A$. We denote $\mathrm{u}\eta:=\mathrm{u}_{F}\eta$ and $\mathrm{u}_{a}\eta:=\mathrm{u}_{\left\{a\right\}}\eta$, for $a\in F$. Let us reproduce some basic facts about this modification.

\begin{proposition}[\cite{erz}, Proposition 3.1]\label{unba}Let $A\subset F$ and let $\eta$ be a locally solid additive convergence on $F$. Then:
\item[(i)] $u_{A}\eta$ is the weakest locally solid additive convergence on $F$, which is stronger than (in fact coincides with) $\eta$ on $\left[0_{F},\left|a\right|\right]$, for every $a\in A$.
\item[(ii)] $u_{A}\eta=u_{I\left(A\right)}\eta$; if $\eta$ is idempotent, then $u_{A}\eta=u_{\overline{I\left(A\right)}}\eta$.
\end{proposition}

It follows that the correspondence $A,\eta\mapsto \mathrm{u}_{A}\eta$ is very far from being injective. In this section we revisit some results in the literature which either relax the definition of the unbounded modification, or explore the limitations of such a relaxation. In particular, some results available for locally solid topologies are true for idempotent locally solid additive convergences.

\begin{proposition}[cf. \cite{taylor}, Theorem 9.6]
If $\eta$ is an order continuous idempotent locally solid convergence and $E$ is an order dense ideal, then $\mathrm{u}_{E}\eta=\mathrm{u}\eta$.
\end{proposition}
\begin{proof}
It is easy to see that $E$ is dense with respect to order convergence, hence $\mathrm{u}_{E}\eta=\mathrm{u}_{\overline{E}_{\eta}}\eta=\mathrm{u}_{F}\eta=\mathrm{u}\eta$.
\end{proof}

Next, let us somewhat improve \cite[Proposition 3.3]{erz} and clean up its proof. The following lemmas are contained there already.

\begin{lemma}\label{l1}
Let $\eta$ be an additive locally solid additive convergence on $F$. If $\left(f_{\alpha}\right)_{\alpha\in A}\subset F$ and $e,f\in F$ are such that $f_{\alpha}\wedge e\xrightarrow[]{\eta} f\wedge e$, then $\left(e-f\right)^{+}\wedge \left|f_{\alpha}-f\right|\xrightarrow[]{\eta}0_{F}$.
\end{lemma}

\begin{lemma}\label{l12}
For every $e,f\in F_{+}$ we have $e\wedge\left(ne-f\right)^{+} \xrightarrow[]{\|\cdot\|_{f}}e$ (and so $e\wedge\left(ne-f\right)^{+} \xrightarrow[]{\mathrm{ru}}e$).
\end{lemma}

\begin{proposition}
Let $\eta$ be an idempotent linear locally solid convergence on $F$, let $E\subset F$ be an ideal, and let $G\subset E_{+}$ be such that $\overline{I\left(G\right)}_{\eta}=E$. Then, $0_{F}\le f_{\alpha}\xrightarrow[]{\mathrm{u}_{E}\eta} f\in F_{+}$ iff $f_{\alpha}\wedge ng \xrightarrow[]{\eta} f\wedge ng$, for every $g\in G$ and $n\in\N$.
\end{proposition}
\begin{proof}
Necessity: If $f_{\alpha}\xrightarrow[]{\mathrm{u}_{E}\eta} f$, then by continuity of the lattice operations we have $f_{\alpha}\wedge ng \xrightarrow[]{\mathrm{u}_{E}\eta} f\wedge ng$, for every $g\in G$ and $n\in\N$. It is left to recall that $\mathrm{u}_{E}\eta$ agrees with $\eta$ on each of $\left[0_{F},ng\right]$.\medskip

Sufficiency: According to \cite[Proposition 2.4]{erz} the set\linebreak $H:=\left\{h\in F:~ \left|h\right|\wedge \left|f_{\alpha}-f\right|\xrightarrow[]{\eta} 0_{F}\right\}$ is a $\eta$-closed ideal. By Lemma \ref{l1} this ideal contains $\left(ng-f\right)^{+}$, for every $g\in G$ and $n\in\N$. Since $\eta$ is linear, it is weaker than relative uniform convergence  (see \cite[Proposition 5.1]{ectv}), and so $H$ is relatively uniformly closed. As $H\ni g\wedge\left(ng-f\right)^{+} \xrightarrow[]{\mathrm{ru}}g$, it follows that $g\in H$, for every $g\in G$. Thus, $E=\overline{I\left(G\right)}_{\eta}\subset H$ and so $f_{\alpha}\xrightarrow[]{\mathrm{u}_{E}\eta} f$.
\end{proof}

\begin{corollary}
If $\eta$ is linear idempotent and $h\in F_{+}$ is a topological unit (i.e. $\overline{F_{h}}_{\eta}=F$), then for a net $\left(f_{\alpha}\right)_{\alpha\in A}\subset F_{+}$ we have $f_{\alpha}\xrightarrow[]{\mathrm{u}\eta} f\in F_{+}$ iff $f_{\alpha}\wedge nh \xrightarrow[]{\eta} f\wedge nh$, for every $n\in\N$.
\end{corollary}

We will now generalize some results from \cite[Section 3]{kt}, where the opposite problem is considered. Namely, what can be said about $\eta,\theta,E,H$ if $\mathrm{u}_{E}\theta\ge \mathrm{u}_{H}\eta$? Let us start with a fact similar to Lemma \ref{l12} (the proof is contained in the proof of \cite[Lemma 11.4]{ectv}).

\begin{lemma}\label{unu1}For every $e,f\in F_{+}$ we have $f\wedge ne\xrightarrow[]{\mathrm{u}_{e}\|\cdot\|_{f}}f$, and so $f\wedge ne\xrightarrow[]{\mathrm{u}_{e}\mathrm{ru}}f$.
\end{lemma}

\begin{proposition}\label{unu}Let $E\subset F$ be an ideal, and let $f\in F_{+}$. Then $\mathrm{u}_{E}\|\cdot\|_{f}\ge \mathrm{u}_{f}\eta$ $\Rightarrow$ $f\in\overline{E}^{1}_{\eta}$ $\Rightarrow$ $\mathrm{u}_{E}\|\cdot\|_{f}\ge \mathrm{u}_{f}\left(\|\cdot\|_{f}\eta\right)$. In particular, if $\eta$ is idempotent and linear, then $\mathrm{u}_{E}\|\cdot\|_{f}\ge \mathrm{u}_{f}\eta$ $\Leftrightarrow$ $f\in\overline{E}^{1}_{\eta}=\overline{E}_{\eta}$.
\end{proposition}
\begin{proof}
Let us prove the first implication. Fix $e\in E_{+}$. The set $\left[0_{F},f\right]\cap E$ is an increasing net which contains $\left(f\wedge ne\right)_{n\in\N}$. By Lemma \ref{unu1} the latter sequence $\mathrm{u}_{e}\|\cdot\|_{f}$-converges to $f$, and so $\left[0_{F},f\right]\cap E\xrightarrow[]{\mathrm{u}_{e}\|\cdot\|_{f}}f$, according to Lemma \ref{mono}. Since $e$ was arbitrary, we conclude that $\left[0_{F},f\right]\cap E\xrightarrow[]{\mathrm{u}_{E}\|\cdot\|_{f}}f$.

As $u_{E}\|\cdot\|_{f}\ge u_{f}\eta$, it follows that $\left[0_{F},f\right]\cap E\xrightarrow[]{\mathrm{u}_{f}\eta}f$. By part (i) of Proposition \ref{unba} $\mathrm{u}_{f}\eta$ coincides with $\eta$ on $\left[0_{F},f\right]$, and so $\left[0_{F},f\right]\cap E\xrightarrow[]{\eta}f$. Thus, $f\in\overline{E}^{1}_{\eta}$.\medskip

We now prove the second implication. Assume that $0_{F}\le f_{\alpha}\xrightarrow[]{\mathrm{u}_{E}\|\cdot\|_{f}}0_{F}$, then for every $e\in \left[0_{F},f\right]\cap E$ we have $e\wedge f_{\alpha}\xrightarrow[]{\|\cdot\|_{f}}0_{F}$. It follows that $f\wedge f_{\alpha}\le e\wedge f_{\alpha} + f-e\xrightarrow[\alpha]{\|\cdot\|_{f}}f-e\xrightarrow[{e\in \left[0_{F},f\right]\cap E}]{\eta}0_{F}$, where the second convergence follows from Corollary \ref{psola1}. Thus, by Proposition \ref{prod} we conclude that $f_{\alpha}\xrightarrow[]{\mathrm{u}_{f}\left(\|\cdot\|_{f}\eta\right)}0_{F}$.\medskip

For the last claim, observe that if $\eta$ is idempotent and linear, we have $\|\cdot\|_{f}\eta\ge \eta^{2}=\eta$.
\end{proof}

\begin{theorem}\label{unbc}
Let $\eta,\theta$ be locally solid additive convergences on $F$ and let $E,H\subset F$ be ideals. Then:
\item[(i)] If $\theta$ is linear and $\mathrm{u}_{E}\theta\ge \mathrm{u}_{H}\eta$, then $H\subset \overline{E}^{1}_{\eta}$ and $\left.\theta\right|_{\left[0_{F},e\right]}\ge \left.\eta\right|_{\left[0_{F},e\right]}$, for every $e\in E\cap H$.
\item[(ii)] If $\eta$ is idempotent, then $\mathrm{u}_{E}\theta\ge \mathrm{u}_{H}\eta$ whenever $H\subset \overline{E}^{1}_{\eta}=\overline{E}_{\eta}$ and $\left.\theta\right|_{\left[0_{F},e\right]}\ge \left.\eta\right|_{\left[0_{F},e\right]}$, for every $e\in E\cap H$.
\end{theorem}
\begin{proof}
(i): Since $\theta$ is linear, for every $h\in H$ we have that $\mathrm{u}_{E}\|\cdot\|_{h}\ge \mathrm{u}_{E}\theta\ge \mathrm{u}_{H}\eta\ge \mathrm{u}_{h}\eta$. Proposition \ref{unu} now guarantees that $h\in\overline{E}^{1}_{\eta}$.\medskip

For every $e\in E\cap H$, $\eta$ and $\theta$ coincide with $\mathrm{u}_{H}\eta$ and $\mathrm{u}_{E}\theta$, respectively, on $\left[0_{F},e\right]$, which justifies the second conclusion.\medskip

(ii): Our assumption together with the extremal property of the unbounded convergence discussed in part (i) of Proposition \ref{unba} imply that $\mathrm{u}_{E\cap H}\theta\ge \mathrm{u}_{E\cap H}\eta$. It then follows from Proposition \ref{adh} that $H\subset \overline{E}_{\eta}\cap \overline{H}_{\eta}=\overline{E\cap H}_{\eta}$. Thus, using part (ii) of Proposition \ref{unba} we conclude that $\mathrm{u}_{H}\eta\le \mathrm{u}_{\overline{E\cap H}_{\eta}}\eta=\mathrm{u}_{E\cap H}\eta\le \mathrm{u}_{E\cap H}\theta\le \mathrm{u}_{E}\theta$.
\end{proof}

\begin{corollary}
If $\eta$ is linear and idempotent and $A,B\subset F$, then $\mathrm{u}_{A}\eta\ge \mathrm{u}_{B}\eta$ iff $B\subset \overline{I\left(A\right)}$.
\end{corollary}

Note that $\mathrm{u}_{E}\theta\ge \mathrm{u}_{H}\eta$ can be only true for sequences (as opposed to all nets). For such an occasion we have the following generalization of \cite[Theorem 3.3]{kt}.

\begin{proposition}
Let $\eta,\theta$ be locally solid additive convergences on $F$ and let $E,H\subset F$ be ideals. Assume furthermore that $\theta$ is linear and $H$ is countably generated. If $\mathrm{u}_{H}\theta\ge \mathrm{u}_{E}\eta$ on sequences, then $E$ is contained in the sequential $\eta$-adherence of $H$.
\end{proposition}
\begin{proof}
We may assume that $H=I\left(\left\{h_{n}\right\}_{n\in\N}\right)$, where $\left(h_{n}\right)_{n\in\N}$ is increasing. Fix $e\in E$. From Lemma \ref{unu1}, for $n\ge m$ we have that $h_{m}\wedge \left(e-nh_{n}\right)^{+}\le h_{m}\wedge \left(e-nh_{m}\right)^{+}\xrightarrow[n]{\mathrm{ru}} 0_{F}$, hence $h_{m}\wedge \left(e-nh_{n}\right)^{+}\xrightarrow[n]{\theta} 0_{F}$, for every $m\in\N$. It then follows that $\left(e-nh_{n}\right)^{+}\xrightarrow[]{\mathrm{u}_{H}\theta} 0_{F}$, therefore $e\wedge nh_{n}\xrightarrow[]{\mathrm{u}_{H}\theta}e$, which according to our assumption yields $\left[0_{F},e\right]\cap H\ni e\wedge nh_{n}\xrightarrow[]{\mathrm{u}_{E}\eta}e$. As $\eta$ and $\mathrm{u}_{E}\eta$ agree on $\left[0_{F},e\right]$, we conclude that $e$ is in the sequential $\eta$-adherence of $H$.
\end{proof}

Let $E\subset F$ be a sublattice, and let $H\subset E$ be an ideal of $F$. For an additive locally solid convergence $\eta$ on $F$ we have $\left.\left(\mathrm{u}_{H}\eta\right)\right|_{E}=\mathrm{u}_{H}\left(\left.\eta\right|_{E}\right)$, i.e. in this case restricting to $E$ commutes with unbounding by $H$. However, if $H\not\subset E$, in general we only have $\left.\left(\mathrm{u}_{H}\eta\right)\right|_{E}\ge\mathrm{u}_{E\cap H}\left(\left.\eta\right|_{E}\right)$. In particular, for the case $H=F$, in general we only have $\left.\left(\mathrm{u}\eta\right)\right|_{E}\ge\mathrm{u}\left(\left.\eta\right|_{E}\right)$. If the latter inequality is in fact equality, we will say that $\eta$ \emph{commutes} with $E$. Some sufficient conditions for this to occur were presented in \cite[Section 4]{kmt} and \cite[Lemma 3.4]{taylor}. We extend them in the following result.

\begin{proposition}\label{latcom}Let $\eta$ be a locally solid additive convergence on $F$ and let $E\subset F$ be a sublattice. Then:
\item[(i)] If $\eta$ commutes with $E$, and $H\subset E$ is a sublattice such that $\left.\eta\right|_{E}$ commutes with $H$, then $\eta$ commutes with $H$.
\item[(ii)] If $I\left(E\right)$ is a projection band, then $E$ commutes with $\eta$.
\item[(iii)] If $\eta$ is idempotent and $\overline{I\left(E\right)+E^{d}}_{\eta}=F$, then $E$ commutes with $\eta$.
\item[(iv)] If $\eta$ is idempotent and order continuous, then $E$ commutes with $\eta$.
\end{proposition}
\begin{proof}
(i) is proven by unpacking the definitions. (ii) is proven similarly to (iii), and (iv) follows from (iii), since $I\left(E\right)+E^{d}$ is always dense with respect to order convergence.\medskip

(iii): If $0_{E}\le e_{\alpha}\xrightarrow[]{\mathrm{u}\left(\left.\eta\right|_{E}\right)} 0_{E}$, then from \cite[Proposition 2.4]{erz} the set $H:=\left\{h\in F,~ \left|h\right|\wedge e_{\alpha}\xrightarrow[]{\eta} 0_{F}\right\}$ is a $\eta$-closed ideal which contains $E$ as well as $E^{d}$. From our assumption it follows that $H=F$, and so $e_{\alpha}\xrightarrow[]{\mathrm{u}\eta} 0_{F}$.
\end{proof}

Note that the conditions (ii) and (iii) are somewhat close to the necessary condition for $E$ to commute with $\eta$ established in \cite[Theorem 11.7]{ectv}: if $\eta$ is linear, complete and commutes with an ideal $E$, then $\overline{E}^{1}_{\eta}$ is a projection band.

\section{Acknowledgements}

The author wants to thank Vladimir Troitsky for general support and many valuable discussions on the topic of this paper. Jan Harm van der Walt gets credit for an idea used in the proof of Theorem \ref{weakest}. Some of the results of the paper were obtained during the author's visit to Sichuan University and Southwestern University of Finance and Economics (Chengdu, China). Many thanks to Yang Deng, Xingni Jiang and Marcel de Jeu for making it possible.

\end{document}